\documentclass[11pt]{article}
\usepackage[T1]{fontenc}
\usepackage[utf8]{inputenc}
\usepackage{lmodern}
\usepackage{geometry}
\usepackage{amsmath,amssymb,amsthm}
\usepackage{booktabs}
\usepackage{array}
\usepackage{adjustbox}
\usepackage{hyperref}
\usepackage{microtype}
\usepackage{enumitem}
\hypersetup{colorlinks=true,linkcolor=blue,citecolor=blue,urlcolor=blue}
\setlist[enumerate]{itemsep=0.2em,topsep=0.4em}
\newcommand{\Rect}{\mathrm{Rect}}
\newcommand{\Att}{\mathrm{Att}}
\newcommand{\Corr}{\mathrm{Corr}}

\theoremstyle{plain}
\newtheorem{proposition}{Proposition}[section]
\newtheorem{lemma}[proposition]{Lemma}
\newtheorem{problem}[proposition]{Problem}
\theoremstyle{definition}
\newtheorem{definition}[proposition]{Definition}
\theoremstyle{remark}
\newtheorem{remark}[proposition]{Remark}
\newtheorem{corollary}[proposition]{Corollary}

\title{Boundary Framework, Rear Morphology, and Rectangular Ears in the Partition Graph}
\author{Fedor B. Lyudogovskiy}
\date{}

\begin{document}
\maketitle

\begin{abstract}
We study the outer geometry of the partition graph $G_n$, focusing on its canonical front-and-side framework, the family of nontrivial rectangular partitions, and the rear structures suggested by the visible global geometry of the graph. We formalize the boundary framework
\[
\mathcal B_n=\mathcal M_n\cup \mathcal L_n\cup \mathcal R_n,
\]
where $\mathcal M_n$ is the main chain and $\mathcal L_n,\mathcal R_n$ are the left and right side edges. We then isolate the nontrivial rectangular family
\[
\Rect^\ast(n)=\{(a^b):ab=n,\ a,b\ge 2\}
\]
as a canonical discrete family marking the rear part of the graph.

We prove that every nontrivial rectangular vertex $\rho=(a^b)$ has degree $2$, has exactly two explicitly described neighbors, and lies in a unique triangle of $G_n$. This leads naturally to the notions of a rectangular ear, its attachment pair, and its support edge. We prove that $\Rect^\ast(n)$ is an independent set in $G_n$, so the weak rectangular contour is not a graph-theoretic chain but rather a discrete rear marker family. We also show that for every genuinely rear rectangular ear, namely for $a,b\ge 3$, its support edge is contained in a tetrahedral configuration of the clique complex $K_n=\mathrm{Cl}(G_n)$.

To organize the interaction between different ears, we introduce support zones, support distances, and support corridors between attachment pairs. This provides a clean language for the study of rear communication through the ambient graph. We also record a natural divisor-theoretic indexing of the rectangular family, in which the trivial divisors correspond to the two antennas $(n)$ and $(1^n)$, and the nontrivial divisors to the rectangular roots. The paper closes with a computational atlas combining a small full-graph layer and a larger local/divisor layer, and with a collection of open problems concerning support-zone connectivity, inter-ear corridors, and the broader problem of canonical rear contours in $G_n$.
\end{abstract}


\section{Introduction}

Let $G_n$ be the partition graph of $n$: its vertices are the integer partitions of $n$, and two partitions are adjacent if one is obtained from the other by a single elementary unit transfer, followed by reordering.

Several complementary aspects of the geometry of $G_n$ were developed in earlier papers of this series. The local transfer geometry, together with the corresponding degree and local simplex data, was analyzed in \cite{Lyudogovskiy-Local}. A subsequent programmatic paper \cite{Lyudogovskiy-Growing} described $G_n$ as a growing discrete geometric object and introduced a broader geometric vocabulary, including the framework, the self-conjugate axis, the central region, the spine, anisotropy, and comparative growth. The axial and simplex layers were then developed further in \cite{Lyudogovskiy-Axial,Lyudogovskiy-Simplex}. For related topological structure of the clique complex, see also \cite{Lyudogovskiy-Homotopy}.

The present paper studies the \emph{outer morphology} of $G_n$. More specifically, we examine the canonical front-and-side framework of the graph, the family of rectangular partitions, and the rear structures suggested by the visible global geometry of the partition graphs. Our aim is to turn the geometric picture of a front framework, side edges, rear appendices, and a rectangular rear profile into a compact combinatorial language that captures part of this outer geometry rigorously.

A central difficulty is that the rear part of $G_n$ is visually prominent but lacks an obvious intrinsic construction. Unlike the main chain, the side edges, or the self-conjugate axis, a ``rear chain'' is not given a priori by a standard intrinsic construction. For this reason, we adopt a restrained approach: when several formalizations are possible, we record them explicitly, isolate the canonical pieces, and clearly distinguish proved results from computational observations and open questions.

The first canonical object considered here is the \emph{boundary framework}
\[
\mathcal B_n=\mathcal M_n\cup \mathcal L_n\cup \mathcal R_n,
\]
formed by the main chain together with the left and right side edges. This object will be viewed as the front-and-side outer skeleton of $G_n$. A second canonical object is the family of nontrivial rectangular partitions
\[
\Rect^\ast(n)=\{(a^b):ab=n,\ a,b\ge 2\},
\]
which provides a natural discrete marker family for the rear part of the graph. This family also carries a natural divisor-theoretic indexing: for each divisor $d\mid n$, the partition $((n/d)^d)$ is rectangular, conjugation interchanges $d$ and $n/d$, the trivial divisors correspond to the two antennas $(n)$ and $(1^n)$, and the nontrivial divisors correspond to the nontrivial rectangular roots. This arithmetic viewpoint will reappear in the larger local atlas. The corresponding weak rectangular contour is simply this set of marker vertices; a one-neighborhood thickening yields a stronger rear candidate better suited for attachment questions.

The main structural point of the paper is that rectangular vertices are not merely visually recognizable. They carry a rigid local graph structure. For every nontrivial rectangular partition $\rho=(a^b)$, we determine its two neighbors explicitly, prove that $\rho$ has degree $2$, and show that it lies in a unique triangle of $G_n$. This leads naturally to the notion of a \emph{rectangular ear}: the rectangular vertex is the terminal vertex of the ear, and the opposite edge is its support edge. We also prove that distinct nontrivial rectangular vertices are never adjacent. Thus the weak rectangular contour is not a graph-theoretic rear chain, but rather a discrete family of rear marker vertices.

A further result concerns the support geometry of these ears. For genuinely rear rectangular vertices, namely those of the form $(a^b)$ with $a,b\ge 3$, we show that the support edge of the associated ear is already contained in a $3$-simplex of the clique complex $K_n=\mathrm{Cl}(G_n)$. In other words, the base of a genuine rear ear is tetrahedrally reinforced. This gives a first rigorous distinction between side behavior and rear behavior in the outer geometry of $G_n$.

These local results naturally lead to the question of how different rectangular ears are connected through the ambient graph. To formalize this, we introduce support zones, support distances, and support corridors between rectangular ears, defined in terms of shortest paths between their attachment pairs after removing the rectangular tips themselves. Within the proved framework, this provides a clean language for inter-ear communication. At the computational level, it leads to an atlas of support zones and rear corridors, where one can compare triangular and tetrahedral behavior along representative connecting routes. In addition to this small full-graph atlas, we also use a larger divisor-based local atlas, which does not require construction of the full graph $G_n$ and serves to display the arithmetic organization and local stability of the rectangular ear family for composite numbers with richer divisor structure.

The paper has two complementary aims. The first is structural: to isolate several canonical objects of the outer morphology of $G_n$ and to prove a compact package of rigorous results about them. The second is organizational: to provide a language in which the front, side, and rear regimes of the partition graph can be compared systematically, in a form compatible with the axial, simplex, and degree stratifications developed in the earlier papers of this series. In particular, the present paper is not intended as a complete theory of graph-theoretic boundary vertices or rear separators. Rather, it is intended as a first formal step toward such a theory.

The main contributions of the paper may be summarized as follows.

\begin{enumerate}
\item We formalize the boundary framework $\mathcal B_n$ as the canonical front-and-side outer skeleton of $G_n$.

\item We isolate the family $\Rect^\ast(n)$ of nontrivial rectangular partitions as a canonical rear marker family and compare weak and strong rear-contour constructions built from it.

\item We prove that every nontrivial rectangular vertex has degree $2$, belongs to a unique triangle, and lies in the degree layer $D_2(n)$ and the simplex layer $L_2(n)$.

\item We introduce rectangular ears, attachment pairs, and support edges, and prove that $\Rect^\ast(n)$ is an independent set in $G_n$.

\item We show that every genuinely rear rectangular ear has a support edge contained in a tetrahedral configuration of the clique complex.

\item We introduce support zones, support distances, and support corridors between rectangular ears as organizing tools for the rear geometry.

\item We locate the rectangular family relative to the framework, the axis, the simplex layers, and the degree layers, thereby linking the outer morphology to the established internal stratifications of the partition graph.
\end{enumerate}

The paper is organized as follows. Section~\ref{sec:framework} introduces the boundary framework, the rectangular family, and the basic rear-contour candidates. Section~\ref{sec:rectangular-ears} establishes the local structure of rectangular vertices, defines rectangular ears and support edges, and proves the tetrahedral support result for genuine rear ears. Section~\ref{sec:support-corridors} introduces support zones, support distances, and support corridors. Section~\ref{sec:atlas} presents a computational atlas combining a small full-graph layer with a larger local/divisor layer. The final section summarizes the picture obtained here and formulates several open problems concerning rear contours, support connectivity, and the broader outer geometry of $G_n$.

\section{Boundary framework and outer reference families}
\label{sec:framework}

We work with the partition graph $G_n$, whose vertices are the integer partitions of $n$, and where adjacency is given by a single elementary unit transfer followed by reordering. All standard notation from the earlier papers in this series is assumed; see in particular \cite{Lyudogovskiy-Local,Lyudogovskiy-Growing,Lyudogovskiy-Axial,Lyudogovskiy-Simplex}. For general background on integer partitions and Ferrers diagrams, see \cite{Andrews,Stanley}.

\subsection{The boundary framework}

Recall that the \emph{main chain} is the canonical shortest path
\[
\mathcal M_n=\{(n-k,1^k):0\le k\le n-1\}
\]
from the one-part partition $(n)$ to the all-ones partition $(1^n)$.

The \emph{left edge} is the family
\[
\mathcal L_n:=\{(n-k,k):1\le k\le \lfloor n/2\rfloor\},
\]
and the \emph{right edge} is its conjugate image
\[
\mathcal R_n:=\mathcal L_n^\ast.
\]

\begin{definition}
The \emph{boundary framework} of $G_n$ is the subgraph
\[
\mathcal B_n:=\mathcal M_n\cup \mathcal L_n\cup \mathcal R_n.
\]
\end{definition}

The set $\mathcal B_n$ is canonically defined, depends only on the intrinsic combinatorics of the partition graph, and is invariant under conjugation. It will be viewed as the basic front-and-side outer skeleton of the graph.

\begin{remark}
In the present paper we do not attempt to characterize all graph-theoretic boundary vertices of $G_n$ in an intrinsic extremal sense. Instead, we use $\mathcal B_n$ as a canonical outer reference family and study how additional outer structures, especially the rectangular rear regime, are positioned relative to it.
\end{remark}

\subsection{Rectangular partitions}

\begin{definition}
A partition of $n$ is called \emph{rectangular} if it has the form
\[
(a^b)
\]
for some integers $a,b\ge 1$ with $ab=n$. Equivalently, its Ferrers diagram is a rectangle.

We write
\[
\Rect(n):=\{(a^b):ab=n,\ a,b\ge 1\}
\]
for the set of all rectangular partitions of $n$, and
\[
\Rect^\ast(n):=\{(a^b):ab=n,\ a,b\ge 2\}
\]
for the set of \emph{nontrivial rectangular partitions}.
\end{definition}

The family $\Rect^\ast(n)$ is nonempty if and only if $n$ is composite. When $n$ is prime, the only rectangular partitions are the two trivial antennas $(n)$ and $(1^n)$.

\subsection{Rear-contour candidates}

The rear profile of $G_n$ suggests several possible formalizations. At this stage we isolate two basic candidates.

\begin{definition}
The \emph{weak rectangular contour} is the vertex set
\[
RC_n^{\mathrm{wk}}:=\Rect^\ast(n).
\]
\end{definition}

This is the smallest natural rear marker family. It is canonical and conjugation-invariant, but it need not be connected and, as will be proved below, does not form an edge-path in general.

\begin{definition}
The \emph{strong rectangular contour} is the induced subgraph on the vertex set
\[
V(RC_n^+):=\Rect^\ast(n)\cup N(\Rect^\ast(n)),
\]
where $N(\Rect^\ast(n))$ denotes the set of all vertices adjacent to at least one nontrivial rectangular partition.
\end{definition}

The weak contour records only the rectangular marker vertices themselves, whereas the strong contour thickens this set by one graph step. The latter is better suited for the discussion of support geometry and ear-like appendices.

\begin{remark}
We do not claim in this paper that either $RC_n^{\mathrm{wk}}$ or $RC_n^+$ is the unique correct formalization of the rear boundary of $G_n$. Rather, these objects provide two natural and canonical rear candidates: a minimal discrete one and a one-neighborhood thickening. A more intrinsic theory of rear chains or rear separators is left for future work.
\end{remark}

\subsection{Auxiliary contour language}

For later comparison with more global complement-based notions, it is convenient to record a standard attachment construction.

\begin{definition}
Let $C_n$ be a chosen vertex set or induced subgraph of $G_n$. If $E$ is a connected component of $G_n\setminus C_n$, then its \emph{attachment locus} along $C_n$ is
\[
\Att_{C_n}(E):=N(E)\cap C_n.
\]
\end{definition}

\begin{remark}
In the proved part of the present paper, the main role is played not by component-based attachment loci, but by the more local notions of attachment pair and support edge introduced in Section~\ref{sec:rectangular-ears}. The contour-based language above will reappear only in the computational and open-problem layers.
\end{remark}

\section{Rectangular vertices, rectangular ears, and support edges}
\label{sec:rectangular-ears}

We retain the notation introduced in Section~\ref{sec:framework}. The purpose of this section is to show that nontrivial rectangular partitions form a distinguished outer marker family with a rigid local structure.

\subsection{The local structure of a rectangular vertex}

Let
\[
\rho=(a^b)\in \Rect^\ast(n),
\qquad a,b\ge 2,
\qquad ab=n.
\]

\begin{lemma}
\label{lem:rect-neighbors}
The vertex $\rho=(a^b)$ has exactly two neighbors in $G_n$, namely
\[
\alpha(\rho):=(a+1,a^{\,b-2},a-1)
\]
and
\[
\beta(\rho):=(a^{\,b-1},a-1,1),
\]
with the obvious simplification
\[
\alpha(\rho)=(a+1,a-1)
\quad\text{when } b=2.
\]
\end{lemma}

\begin{proof}
We use the local transfer description already established in the local morphology paper, but in this special case the outcome is completely explicit.

Since all parts of $\rho=(a^b)$ are equal to $a$, there are only two one-step possibilities.

First, one may transfer one unit from one copy of $a$ to another copy of $a$. After reordering, this yields
\[
(a+1,a^{\,b-2},a-1)=\alpha(\rho).
\]

Second, one may move one unit from one copy of $a$ so as to create a new part of size $1$. After reordering, this yields
\[
(a^{\,b-1},a-1,1)=\beta(\rho).
\]

Thus $\alpha(\rho)$ and $\beta(\rho)$ are both neighbors of $\rho$. Because all parts of $\rho$ are equal, any unit transfer starting from $\rho$ must be of one of the two types listed above: either one unit is moved from one copy of $a$ to another, yielding $\alpha(\rho)$ after reordering, or one unit is removed from a part of size $a$ to create a new part $1$, yielding $\beta(\rho)$. Hence the explicit case analysis is exhaustive, and the two displayed neighbors are the only neighbors of $\rho$. This is also consistent with the degree formula established in the local morphology paper \cite{Lyudogovskiy-Local}.
\end{proof}

\begin{lemma}
\label{lem:rect-adj}
The two neighbors $\alpha(\rho)$ and $\beta(\rho)$ are adjacent.
\end{lemma}

\begin{proof}
Starting from
\[
\alpha(\rho)=(a+1,a^{\,b-2},a-1),
\]
move one unit from the part $a+1$ so as to create a new part $1$. After reordering, this gives
\[
(a^{\,b-1},a-1,1)=\beta(\rho).
\]
Hence $\alpha(\rho)$ and $\beta(\rho)$ are adjacent.
\end{proof}

\begin{proposition}
\label{prop:rect-local}
Let $\rho=(a^b)\in \Rect^\ast(n)$. Then the following hold.

\begin{enumerate}
\item $\deg_{G_n}(\rho)=2$;
\item the neighbors of $\rho$ are exactly $\alpha(\rho)$ and $\beta(\rho)$;
\item the three vertices
\[
\rho,\ \alpha(\rho),\ \beta(\rho)
\]
form a triangle in $G_n$;
\item this is the unique maximal clique of $G_n$ containing $\rho$;
\item consequently,
\[
\rho\in D_2(n)\cap L_2(n).
\]
\end{enumerate}
\end{proposition}

\begin{proof}
Parts (1)--(3) follow from Lemmas~\ref{lem:rect-neighbors} and \ref{lem:rect-adj}.

Since $\rho$ has exactly two neighbors, every clique containing $\rho$ is contained in the set
\[
\{\rho,\alpha(\rho),\beta(\rho)\}.
\]
Because these three vertices already form a triangle, every clique containing $\rho$ is contained in this triangle. Therefore
\[
\{\rho,\alpha(\rho),\beta(\rho)\}
\]
is the unique maximal clique of $G_n$ containing $\rho$. It follows that the local simplex dimension at $\rho$ is equal to $2$, which is precisely the statement that $\rho\in L_2(n)$. The inclusion $\rho\in D_2(n)$ is just part~(1).
\end{proof}

\subsection{Rectangular ears and support edges}

The local motif singled out in Proposition~\ref{prop:rect-local} will be regarded as the basic ear-like configuration carried by a rectangular vertex.

\begin{definition}
\label{def:rect-ear}
Let $\rho\in \Rect^\ast(n)$.

\begin{enumerate}
\item The \emph{rectangular ear rooted at $\rho$} is the rooted $V$-configuration consisting of the vertex $\rho$ and the two edges
\[
\rho\alpha(\rho),\qquad \rho\beta(\rho).
\]

\item The unordered pair
\[
A(\rho):=\{\alpha(\rho),\beta(\rho)\}
\]
is called the \emph{attachment pair} of the ear.

\item The edge
\[
s(\rho):=\alpha(\rho)\beta(\rho)
\]
is called the \emph{support edge} of the ear.

\item The triangle
\[
T(\rho):=\{\rho,\alpha(\rho),\beta(\rho)\}
\]
will be called the \emph{triangular closure} of the ear.
\end{enumerate}
\end{definition}

Thus the rectangular vertex $\rho$ is the terminal vertex of the ear, whereas the support edge $s(\rho)$ is its base inside the ambient graph.

\begin{proposition}
\label{prop:rect-independent}
The set $\Rect^\ast(n)$ is an independent set in $G_n$. Equivalently, no two distinct nontrivial rectangular partitions are adjacent.
\end{proposition}

\begin{proof}
Let $\rho=(a^b)\in \Rect^\ast(n)$. By Lemma~\ref{lem:rect-neighbors}, every neighbor of $\rho$ is equal either to
\[
\alpha(\rho)=(a+1,a^{\,b-2},a-1)
\]
or to
\[
\beta(\rho)=(a^{\,b-1},a-1,1).
\]
Neither of these partitions is rectangular. Indeed, $\alpha(\rho)$ has at least two distinct part sizes, namely $a+1$ and $a-1$. Likewise, $\beta(\rho)$ has at least two distinct part sizes: it contains both $a$ and $1$, and therefore cannot be rectangular. Hence no neighbor of $\rho$ belongs to $\Rect^\ast(n)$. Since $\rho$ was arbitrary, the whole set $\Rect^\ast(n)$ is independent.
\end{proof}

\begin{remark}
\label{rem:weak-discrete}
Proposition~\ref{prop:rect-independent} shows that the weak rectangular contour
\[
RC_n^{\mathrm{wk}}=\Rect^\ast(n)
\]
is a discrete rear marker family rather than a graph-theoretic path. Any connected rear contour must therefore arise from a thickening or from a different construction.
\end{remark}

\subsection{Tetrahedral support for genuine rear ears}

The support edge of a rectangular ear may itself lie in higher-dimensional simplices of the clique complex. For genuinely rear rectangular vertices, that support is already tetrahedral.

\begin{proposition}
\label{prop:rect-tetra}
Let
\[
\rho=(a^b)\in \Rect^\ast(n)
\]
with
\[
a\ge 3,\qquad b\ge 3.
\]
Define
\[
\gamma_1(\rho):=(a+1,a^{\,b-3},(a-1)^2,1)
\]
and
\[
\gamma_2(\rho):=(a+1,a^{\,b-2},a-2,1).
\]
Then the four vertices
\[
\alpha(\rho),\ \beta(\rho),\ \gamma_1(\rho),\ \gamma_2(\rho)
\]
form a $4$-clique in $G_n$. In particular, the support edge $s(\rho)=\alpha(\rho)\beta(\rho)$ is contained in a $3$-simplex of the clique complex $K_n=\mathrm{Cl}(G_n)$.
\end{proposition}

\begin{proof}
We verify the six adjacencies directly.

\smallskip
\noindent
\emph{(i) $\alpha(\rho)\sim \beta(\rho)$.}
This is Lemma~\ref{lem:rect-adj}.

\smallskip
\noindent
\emph{(ii) $\alpha(\rho)\sim \gamma_1(\rho)$.}
Starting from
\[
\alpha(\rho)=(a+1,a^{\,b-2},a-1),
\]
move one unit from one of the parts of size $a$ so as to create a new $1$. After reordering, this yields
\[
(a+1,a^{\,b-3},(a-1)^2,1)=\gamma_1(\rho).
\]

\smallskip
\noindent
\emph{(iii) $\alpha(\rho)\sim \gamma_2(\rho)$.}
Starting from $\alpha(\rho)$, move one unit from the part $a-1$ so as to create a new $1$. After reordering, this yields
\[
(a+1,a^{\,b-2},a-2,1)=\gamma_2(\rho).
\]

\smallskip
\noindent
\emph{(iv) $\beta(\rho)\sim \gamma_1(\rho)$.}
Starting from
\[
\beta(\rho)=(a^{\,b-1},a-1,1),
\]
transfer one unit from one of the parts of size $a$ to another part of size $a$. After reordering, this yields
\[
(a+1,a^{\,b-3},(a-1)^2,1)=\gamma_1(\rho).
\]

\smallskip
\noindent
\emph{(v) $\beta(\rho)\sim \gamma_2(\rho)$.}
Starting from $\beta(\rho)$, transfer one unit from the part $a-1$ to one of the parts of size $a$. After reordering, this yields
\[
(a+1,a^{\,b-2},a-2,1)=\gamma_2(\rho).
\]

\smallskip
\noindent
\emph{(vi) $\gamma_1(\rho)\sim \gamma_2(\rho)$.}
Starting from
\[
\gamma_1(\rho)=(a+1,a^{\,b-3},(a-1)^2,1),
\]
transfer one unit from one part of size $a-1$ to the other part of size $a-1$. After reordering, this yields
\[
(a+1,a^{\,b-2},a-2,1)=\gamma_2(\rho).
\]

Thus every pair among
\[
\alpha(\rho),\ \beta(\rho),\ \gamma_1(\rho),\ \gamma_2(\rho)
\]
is adjacent. Hence these four vertices form a $4$-clique.

When $a=3$ or $b=3$, some displayed part sizes may coincide after reordering, but the same formulas still define valid partitions and the same one-step verifications remain valid.
\end{proof}

\begin{remark}
\label{rem:side-rear-support}
Proposition~\ref{prop:rect-tetra} shows that genuinely rear rectangular ears, namely those with $a,b\ge 3$, are attached along a support edge that already sits inside a tetrahedral configuration. The borderline cases with $\min(a,b)=2$ are more degenerate and will be treated only computationally in the present paper.
\end{remark}

\subsection{Interaction with the framework and the axis}

We now locate the rectangular family relative to the previously established outer and inner reference structures.

\begin{proposition}
\label{prop:rect-intersections}
For every $n\ge 2$, the following hold.

\begin{enumerate}
\item
\[
\Rect^\ast(n)\cap \mathcal M_n=\varnothing.
\]

\item If $n$ is even and $n\ge 4$, then
\[
\Rect^\ast(n)\cap \mathcal L_n=\{(n/2,n/2)\},
\qquad
\Rect^\ast(n)\cap \mathcal R_n=\{(2^{n/2})\};
\]
if $n=2$ or if $n$ is odd, then both intersections are empty.

\item
\[
\Rect^\ast(n)\cap Ax_n=
\begin{cases}
\{(a^a)\}, & \text{if } n=a^2 \text{ is a square},\\[1ex]
\varnothing, & \text{otherwise}.
\end{cases}
\]

\item
\[
\Rect^\ast(n)\subset D_2(n)\cap L_2(n).
\]
\end{enumerate}
\end{proposition}

\begin{proof}
(1) By the explicit description of the main chain,
\[
\mathcal M_n=\{(n-k,1^k):0\le k\le n-1\}.
\]
A nontrivial rectangular partition $(a^b)$ with $a,b\ge 2$ has all parts equal and contains neither a unique large part nor a tail of ones. Hence it cannot lie on $\mathcal M_n$.

\smallskip
\noindent
(2) By definition,
\[
\mathcal L_n=\{(n-k,k):1\le k\le \lfloor n/2\rfloor\}.
\]
A partition of the form $(n-k,k)$ is rectangular if and only if its two parts are equal, that is, if and only if $n-k=k$. Thus
\[
\Rect^\ast(n)\cap \mathcal L_n=
\begin{cases}
\{(n/2,n/2)\}, & n \text{ even and } n\ge 4,\\
\varnothing, & n=2 \text{ or } n \text{ odd}.
\end{cases}
\]
Applying conjugation gives the corresponding statement for $\mathcal R_n$:
\[
(\Rect^\ast(n)\cap \mathcal L_n)^\ast=\Rect^\ast(n)\cap \mathcal R_n,
\]
and for even $n\ge 4$ this gives
\[
(n/2,n/2)^\ast=(2^{n/2}).
\]

\smallskip
\noindent
(3) A rectangular Ferrers diagram is self-conjugate if and only if it is a square. Thus $(a^b)$ is self-conjugate if and only if $a=b$, equivalently $n=a^2$.

\smallskip
\noindent
(4) This is exactly Proposition~\ref{prop:rect-local}(5).
\end{proof}

\begin{corollary}
\label{cor:rect-framework}
The only nontrivial rectangular vertices on the boundary framework are the side-framework intersections
\[
(n/2,n/2)\in \mathcal L_n,
\qquad
(2^{n/2})\in \mathcal R_n
\]
when $n$ is even and $n\ge 4$. These two vertices are conjugate and are distinct except in the degenerate case $n=4$. All other nontrivial rectangular vertices lie off the framework.
\end{corollary}

\begin{proof}
By Proposition~\ref{prop:rect-intersections}(1), no nontrivial rectangular vertex lies on $\mathcal M_n$. By part~(2), the only possible intersections with the side edges are the two stated vertices, and these occur only when $n$ is even.
\end{proof}

\section{Support corridors between rectangular ears}
\label{sec:support-corridors}

The local picture established in Section~\ref{sec:rectangular-ears} suggests that nontrivial rectangular vertices should not be viewed merely as isolated markers. Rather, each such vertex carries a distinguished support edge, and it is natural to ask how these support edges are positioned relative to one another inside the ambient graph. This leads to a second layer of outer morphology: the study of support corridors between rectangular ears.

\subsection{Side ears and genuine rear ears}

\begin{definition}
Let
\[
\rho=(a^b)\in \Rect^\ast(n).
\]

\begin{enumerate}
\item The ear rooted at $\rho$ is called a \emph{side ear} if $\rho\in \mathcal B_n$.

\item The ear rooted at $\rho$ is called a \emph{genuine rear ear} if $\rho\notin \mathcal B_n$.
\end{enumerate}
\end{definition}

By Corollary~\ref{cor:rect-framework}, side ears occur only when $n$ is even and $n\ge 4$, and then the side roots are precisely
\[
(n/2,n/2)\in \mathcal L_n,
\qquad
(2^{n/2})\in \mathcal R_n,
\]
which are distinct except in the degenerate case $n=4$. All remaining nontrivial rectangular ears are genuine rear ears.

\subsection{Support zones and support distances}

From the graph-theoretic point of view, the natural base of a rectangular ear is its support edge rather than its tip.

\begin{definition}
The \emph{support zone} of the rectangular family is the induced subgraph
\[
\Sigma_n:=G_n\Big[\bigcup_{\rho\in \Rect^\ast(n)} A(\rho)\Big],
\]
that is, the induced subgraph on the union of all attachment pairs.
\end{definition}

\begin{definition}
Let $\rho,\sigma\in \Rect^\ast(n)$. The \emph{support distance} between the ears rooted at $\rho$ and $\sigma$ is the quantity
\[
d_{\mathrm{sup}}(\rho,\sigma)
\in \mathbb N_0\cup\{\infty\}
\]
defined by
\[
d_{\mathrm{sup}}(\rho,\sigma)
:=
\min\{d_{G_n\setminus \Rect^\ast(n)}(u,v):u\in A(\rho),\ v\in A(\sigma)\},
\]
whenever at least one path in $G_n\setminus \Rect^\ast(n)$ joins a vertex of $A(\rho)$ to a vertex of $A(\sigma)$, and by
\[
d_{\mathrm{sup}}(\rho,\sigma):=\infty
\]
if no such path exists.
\end{definition}

Thus the rectangular tips themselves are removed, and the distance is measured between the corresponding bases inside the remaining graph. In particular, the value $\infty$ records the possibility that the two attachment pairs lie in different connected components of $G_n\setminus \Rect^\ast(n)$.

\begin{remark}
Other choices are possible: one could measure distance in $G_n$, in $G_n\setminus \Rect^\ast(n)$, or inside the support zone $\Sigma_n$. In the present paper we adopt the middle option as the basic one. The convention above makes the definition well posed even when the relevant support-zone data are disconnected.
\end{remark}

\begin{definition}
Let $\rho,\sigma\in \Rect^\ast(n)$.

\begin{enumerate}
\item If $d_{\mathrm{sup}}(\rho,\sigma)<\infty$, a \emph{support geodesic} from $\rho$ to $\sigma$ is a shortest path in $G_n\setminus \Rect^\ast(n)$ joining some vertex of $A(\rho)$ to some vertex of $A(\sigma)$.

\item Any chosen support geodesic will be called a \emph{support corridor} between $\rho$ and $\sigma$.

\item If $d_{\mathrm{sup}}(\rho,\sigma)<\infty$, the union of all support geodesics from $\rho$ to $\sigma$ will be denoted by
\[
\Corr_{\mathrm{all}}(\rho,\sigma)
\]
and called the \emph{full support-corridor set}. If $d_{\mathrm{sup}}(\rho,\sigma)=\infty$, we set
\[
\Corr_{\mathrm{all}}(\rho,\sigma):=\varnothing.
\]
\end{enumerate}
\end{definition}

\begin{remark}
The full support-corridor set $\Corr_{\mathrm{all}}(\rho,\sigma)$ need not be connected a priori, since different shortest endpoint-pair geodesics may be disjoint. It may also be empty when the corresponding attachment pairs are disconnected in $G_n\setminus \Rect^\ast(n)$.
\end{remark}

\subsection{Immediate structural consequences}

The next proposition records what follows immediately from Section~\ref{sec:rectangular-ears}.

\begin{proposition}
\label{prop:support-layer}
Let
\[
\rho=(a^b)\in \Rect^\ast(n)
\]
with $a,b\ge 3$, so that the ear rooted at $\rho$ is a genuine rear ear. Then each endpoint of the support edge $s(\rho)$ belongs to a clique of size $4$. In particular,
\[
\alpha(\rho),\ \beta(\rho)\in \bigcup_{r\ge 3} L_r(n),
\]
that is, both vertices lie in simplex layers of local simplex dimension at least $3$.
\end{proposition}

\begin{proof}
By Proposition~\ref{prop:rect-tetra}, the support edge
\[
s(\rho)=\alpha(\rho)\beta(\rho)
\]
is contained in a $4$-clique. Hence both endpoints $\alpha(\rho)$ and $\beta(\rho)$ belong to a $3$-simplex of the clique complex $K_n=\mathrm{Cl}(G_n)$, so their local simplex dimension is at least $3$.
\end{proof}

\begin{corollary}
\label{cor:rear-start}
Every support corridor issuing from a genuine rear ear starts at a vertex whose local simplex dimension is at least $3$.
\end{corollary}

\begin{proof}
A support corridor starts at one of the two vertices of the attachment pair $A(\rho)=\{\alpha(\rho),\beta(\rho)\}$, and Proposition~\ref{prop:support-layer} shows that both of these vertices have local simplex dimension at least $3$.
\end{proof}

\begin{remark}
This provides a first rigorous distinction between side and rear behavior. A genuine rear ear is not attached merely at the triangular level: its base already lies in a tetrahedral part of the clique geometry. Whether this higher-dimensional support persists along long inter-ear corridors is a separate question and will be addressed only computationally in the present paper.
\end{remark}

\subsection{Conjugation symmetry}

Since conjugation is an automorphism of $G_n$, all support constructions inherit a natural symmetry.

\begin{proposition}
\label{prop:support-conj}
Let $\rho,\sigma\in \Rect^\ast(n)$. Then:

\begin{enumerate}
\item
\[
A(\rho^\ast)=A(\rho)^\ast,
\qquad
s(\rho^\ast)=s(\rho)^\ast;
\]

\item
\[
d_{\mathrm{sup}}(\rho,\sigma)
=
d_{\mathrm{sup}}(\rho^\ast,\sigma^\ast);
\]

\item conjugation sends support geodesics from $\rho$ to $\sigma$ bijectively to support geodesics from $\rho^\ast$ to $\sigma^\ast$, and therefore
\[
\Corr_{\mathrm{all}}(\rho^\ast,\sigma^\ast)
=
\Corr_{\mathrm{all}}(\rho,\sigma)^\ast.
\]
\end{enumerate}
\end{proposition}

\begin{proof}
Conjugation is a graph automorphism of $G_n$ and preserves the set $\Rect^\ast(n)$. Therefore it sends each rectangular ear rooted at $\rho$ to the conjugate ear rooted at $\rho^\ast$, each attachment pair to the conjugate attachment pair, and each support edge to the conjugate support edge. It also preserves the graph $G_n\setminus \Rect^\ast(n)$ and therefore preserves both path existence and path nonexistence between attachment pairs. In particular, the value $d_{\mathrm{sup}}(\rho,\sigma)$ is finite if and only if $d_{\mathrm{sup}}(\rho^\ast,\sigma^\ast)$ is finite. When the distance is finite, graph distances and shortest paths are preserved by automorphisms, so the stated equalities follow.
\end{proof}

\subsection{Open structural questions in the proved framework}

The definitions above isolate a family of precise questions.

\begin{problem}
Describe the support zone $\Sigma_n$ as an induced subgraph of $G_n$. In particular, determine whether $\Sigma_n$ is connected, and how its connected components relate to the side framework and to the self-conjugate axis.
\end{problem}

\begin{problem}
Given two rectangular ears rooted at $\rho$ and $\sigma$, determine or estimate the support distance $d_{\mathrm{sup}}(\rho,\sigma)$ in terms of the factor data of $\rho$ and $\sigma$.
\end{problem}

\begin{problem}
Determine which edges of a support corridor lie only in triangles, which lie in tetrahedra, and whether higher-dimensional simplices occur systematically along long rear corridors.
\end{problem}

\begin{problem}
Characterize the transition from side ears to genuine rear ears. In particular, determine whether every support corridor between two genuine rear ears must pass through a region of lower simplex dimension, or whether purely tetrahedral rear corridors can occur.
\end{problem}

\section{Computational atlas of rectangular ears, support zones, and rear corridors}
\label{sec:atlas}

The results proved in the preceding sections isolate several canonical objects of the outer geometry of $G_n$: the boundary framework $\mathcal B_n$, the nontrivial rectangular family $\Rect^\ast(n)$, the associated rectangular ears, and the support structures carried by their attachment pairs and support edges. The purpose of the present section is to illustrate how these objects appear in concrete partition graphs.

The role of the atlas is illustrative rather than foundational. No statement proved earlier depends on the computations recorded here. Rather, the atlas serves three specific purposes:
\begin{enumerate}
\item to visualize the placement of rectangular ears relative to the boundary framework;
\item to display the geometry of support zones and representative support corridors;
\item to compare the newly introduced outer objects with the previously established axis, spine, simplex layers, and degree layers.
\end{enumerate}

Accordingly, the atlas is not intended as an exhaustive data repository, but as a structured collection of representative outer-morphology patterns.

\subsection{Range and visualization scheme}

For the purposes of the present paper, it is sufficient to display moderate ranges of $n$ large enough to exhibit genuine rear ears and inter-ear corridors, yet small enough to remain visually interpretable when the full graph is used.

In the computational material reported here, we use two complementary ranges. For the full-graph atlas, where support zones and support corridors are studied globally, we use
\[
n\in\{8,9,10,12\}.
\]
For the larger local/divisor atlas, where only rectangular roots and their explicit support neighborhoods are needed, we use
\[
n\in\{60,64,72,81\}.
\]

The first range was chosen because it is small enough for full-graph inspection but already exhibits side ears, genuine rear ears, and nontrivial support-corridor behavior. The second range was chosen because these values have richer divisor structure, including both square and nonsquare cases, and therefore produce a more varied family of rectangular roots while still allowing completely explicit local analysis.

For each selected value of $n$, the atlas compares some or all of the following highlighted structures:
\begin{enumerate}
\item the boundary framework $\mathcal B_n$;
\item the nontrivial rectangular roots $\Rect^\ast(n)$;
\item the support edges $s(\rho)$ for $\rho\in\Rect^\ast(n)$;
\item the support zone
\[
\Sigma_n=G_n\Big[\bigcup_{\rho\in\Rect^\ast(n)}A(\rho)\Big];
\]
\item representative support geodesics between distinct ears;
\item selected overlays with $Ax_n$, $Sp_n$, degree layers, or simplex layers.
\end{enumerate}

The computational data were obtained directly from the partition graphs generated by the unit-transfer adjacency rule. For the full-graph range $n\in\{8,9,10,12\}$, we constructed $G_n$, computed the local simplex dimension of a vertex from the size of the largest clique containing it, and evaluated support distances in the graph $G_n\setminus \Rect^\ast(n)$ by shortest-path search between the relevant attachment pairs. For the larger local/divisor range $n\in\{60,64,72,81\}$, no full graph was needed: the rectangular roots and their local ear data were obtained directly from the explicit formulas of Section~\ref{sec:rectangular-ears}.

Each displayed table is included to support a concrete structural point. We do not aim at a full visual census of all partitions and all paths.

\subsection{Rectangular ears relative to the framework}

The first full-graph layer of the atlas records the graph-theoretic data associated with the boundary framework $\mathcal B_n$, the roots of all nontrivial rectangular ears, and the corresponding support edges. This layer is used to make the following points explicit.

\begin{enumerate}
\item The weak rectangular contour
\[
RC_n^{\mathrm{wk}}=\Rect^\ast(n)
\]
appears as a sparse family of marker vertices rather than as a connected rear chain.

\item The only rectangular roots lying on the boundary framework are the side-framework intersections
\[
(n/2,n/2)\in \mathcal L_n,
\qquad
(2^{n/2})\in \mathcal R_n,
\]
which occur only when $n$ is even and coincide only in the degenerate case $n=4$.

\item All remaining rectangular roots lie strictly off the framework and therefore represent genuinely rear ears.

\item Even at the visual level, the support edges $s(\rho)$ organize the rectangular family more effectively than the roots alone.
\end{enumerate}

A compact summary table for the first atlas range is given below.

\begin{table}[ht]
\centering
\caption{Rectangular ears and support structures in the full-graph atlas}
\label{tab:atlas-ears}
\begin{adjustbox}{max width=\textwidth}
\begin{tabular}{ccccccccc}
\toprule
$n$ & $\rho$ & type & $\alpha(\rho)$ & $\beta(\rho)$ & $\dim_{\mathrm{loc}}(\rho)$ & $\dim_{\mathrm{loc}}(\alpha)$ & $\dim_{\mathrm{loc}}(\beta)$ & remarks \\
\midrule
8  & $(4,4)$         & side ear         & $(5,3)$                 & $(4,3,1)$                 & 2 & 2 & 3 & framework \\
8  & $(2^4)$         & side ear         & $(3,2,2,1)$             & $(2,2,2,1,1)$             & 2 & 3 & 2 & conjugate of $(4,4)$ \\
9  & $(3^3)$         & genuine rear ear & $(4,3,2)$               & $(3,3,2,1)$               & 2 & 3 & 3 & self-conjugate \\
10 & $(5,5)$         & side ear         & $(6,4)$                 & $(5,4,1)$                 & 2 & 2 & 3 & framework \\
10 & $(2^5)$         & side ear         & $(3,2,2,2,1)$           & $(2,2,2,2,1,1)$           & 2 & 3 & 2 & framework \\
12 & $(6,6)$         & side ear         & $(7,5)$                 & $(6,5,1)$                 & 2 & 2 & 3 & framework \\
12 & $(4^3)$         & genuine rear ear & $(5,4,3)$               & $(4,4,3,1)$               & 2 & 3 & 3 & rear \\
12 & $(3^4)$         & genuine rear ear & $(4,3,3,2)$             & $(3,3,3,2,1)$             & 2 & 3 & 3 & rear \\
12 & $(2^6)$         & side ear         & $(3,2,2,2,2,1)$         & $(2,2,2,2,2,1,1)$         & 2 & 3 & 2 & framework \\
\bottomrule
\end{tabular}
\end{adjustbox}
\end{table}

\paragraph{Sparse rear markers.}
The computed full-graph data for $n=8,9,10,12$ confirm that the weak rectangular contour
\[
RC_n^{\mathrm{wk}}=\Rect^\ast(n)
\]
appears as a sparse family of marker vertices rather than as a path-like contour. Even when several nontrivial rectangular roots are present, as in the case $n=12$, they do not organize themselves into a direct rear chain at the level of roots alone.

\paragraph{Side ears versus genuine rear ears.}
In the computed examples, the distinction between side ears and genuine rear ears is already visible at the global level. For $n=8$ and $n=10$, only side ears occur; for $n=9$, the unique nontrivial rectangular root $(3^3)$ is a self-conjugate genuine rear ear; and for $n=12$, two side ears coexist with two genuinely rear ears. Thus the rectangular family is not uniformly rear: in the even case, its two extremal members mediate between the side framework and the genuinely rear regime.

\paragraph{From roots to support geometry.}
The computed examples suggest that the first nontrivial rear organization does not appear at the level of the roots $\rho\in\Rect^\ast(n)$ themselves, but only after one passes to their support edges $s(\rho)$. This becomes especially clear for $n=12$, where the roots remain sparse while the support edges already begin to organize a visible rear support pattern.

\begin{remark}
Already at this stage, the tabulated full-graph data make clear that the roots $\rho\in\Rect^\ast(n)$ themselves are too sparse to form a geometric contour in the literal graph-theoretic sense. The first nontrivial rear organization arises only after one passes to support edges and support zones.
\end{remark}

\subsection{Local support geometry}

The second computational layer focuses on the immediate neighborhood of a rectangular ear. The simplest useful local object consists of the rooted ear itself, its triangular closure, and, in the genuine rear case $a,b\ge 3$, the explicit tetrahedral reinforcement provided by Proposition~\ref{prop:rect-tetra}.

For a rectangular root
\[
\rho=(a^b),\qquad a,b\ge 3,
\]
the relevant local configuration consists of the vertices
\[
\rho,\ \alpha(\rho),\ \beta(\rho),\ \gamma_1(\rho),\ \gamma_2(\rho),
\]
with the support edge $s(\rho)=\alpha(\rho)\beta(\rho)$ clearly distinguished from the two edges incident to $\rho$.

This local package separates three levels of structure:
\begin{enumerate}
\item the \emph{ear itself}, i.e. the rooted $V$-configuration at $\rho$;
\item its triangular closure $T(\rho)$;
\item the tetrahedral support around $s(\rho)$ in the genuine rear case.
\end{enumerate}

\paragraph{Local tetrahedral reinforcement.}
The computed local configurations around representative roots $\rho=(a^b)$ with $a,b\ge 3$ provide a geometric illustration of Proposition~\ref{prop:rect-tetra}: the ear itself is a rooted $V$-configuration, its triangular closure is immediate, and the support edge sits inside a tetrahedral neighborhood. In the computed range, this is already visible for $(3^3)$, $(4^3)$, and $(3^4)$, whereas the support edges of the side ears remain only triangular.

This local geometry shows that the rectangular family is more than a visually distinguished subset of partitions: it carries a rigid and explicitly describable support structure.

\subsection{Divisor indexing and local ear geometry for larger \texorpdfstring{$n$}{n}}
\label{sec:large-local-atlas}

The small full-graph atlas for $n=8,9,10,12$ is sufficient to display the basic corridor phenomena, but it is not the only useful computational layer for the present paper. Since the local structure of a rectangular ear is completely explicit, one can also study the rectangular family for substantially larger values of $n$ without constructing the full graph $G_n$.

This second computational layer is divisor-theoretic rather than global. Its purpose is to illustrate the arithmetic organization of the rectangular family and to display the local ear geometry for composite numbers with richer divisor structure.

\paragraph{Divisor-theoretic indexing.}
For every divisor $d\mid n$, the partition
\[
\lambda_d:=((n/d)^d)
\]
is rectangular, and conjugation interchanges $d$ and $n/d$. In this dictionary, the two trivial divisors $1$ and $n$ correspond to the two antennas
\[
(n),\qquad (1^n),
\]
while the nontrivial divisors correspond to the nontrivial rectangular roots. In the even case, the divisor pair $2,n/2$ yields exactly the two side-ear roots
\[
(n/2,n/2)\in\mathcal L_n,
\qquad
(2^{n/2})\in\mathcal R_n.
\]
Thus the rectangular family is naturally indexed by the divisor structure of $n$, with conjugation acting as divisor inversion.

To illustrate this arithmetic organization, we record the rectangular roots up to conjugation for the larger values
\[
n\in\{60,64,72,81\}.
\]

\begin{table}[ht]
\centering
\caption{Rectangular roots up to conjugation for selected large $n$}
\label{tab:large-rect-roots}
\small
\begin{adjustbox}{max width=\textwidth}
\begin{tabular}{ccccccc}
\toprule
$n$ & $d$ & $n/d$ & $\rho=((n/d)^d)$ & type & tetrahedral? & remarks \\
\midrule
60 & 2 & 30 & $(30,30)$   & side ear                & no  & framework intersection \\
60 & 3 & 20 & $(20^3)$     & genuine rear ear        & yes & conjugate to $(3^{20})$ \\
60 & 4 & 15 & $(15^4)$     & genuine rear ear        & yes & conjugate to $(4^{15})$ \\
60 & 5 & 12 & $(12^5)$     & genuine rear ear        & yes & conjugate to $(5^{12})$ \\
60 & 6 & 10 & $(10^6)$     & genuine rear ear        & yes & paired with $(6^{10})$ \\
\midrule
64 & 2 & 32 & $(32,32)$    & side ear                & no  & framework intersection \\
64 & 4 & 16 & $(16^4)$     & genuine rear ear        & yes & conjugate to $(4^{16})$ \\
64 & 8 & 8  & $(8^8)$      & self-conjugate rear ear & yes & square root \\
\midrule
72 & 2 & 36 & $(36,36)$    & side ear                & no  & framework intersection \\
72 & 3 & 24 & $(24^3)$     & genuine rear ear        & yes & conjugate to $(3^{24})$ \\
72 & 4 & 18 & $(18^4)$     & genuine rear ear        & yes & conjugate to $(4^{18})$ \\
72 & 6 & 12 & $(12^6)$     & genuine rear ear        & yes & conjugate to $(6^{12})$ \\
72 & 8 & 9  & $(9^8)$      & genuine rear ear        & yes & conjugate to $(8^9)$ \\
\midrule
81 & 3 & 27 & $(27^3)$     & genuine rear ear        & yes & conjugate to $(3^{27})$ \\
81 & 9 & 9  & $(9^9)$      & self-conjugate rear ear & yes & square root \\
\bottomrule
\end{tabular}
\end{adjustbox}
\end{table}

\paragraph{Antennas and ears in the divisor picture.}
Table~\ref{tab:large-rect-roots} makes visible a simple but structurally useful contrast. The trivial divisors $1$ and $n$ correspond to the two antennas, which lie at the extreme front of the graph, whereas the nontrivial divisors correspond to the rectangular roots, which form a rear marker family. Thus antennas and ears are geometrically far apart but arithmetically belong to the same divisor-indexed family. In the even case, the pair $2,n/2$ plays a special transition role by producing exactly the two side-ear roots on the boundary framework.

The local support structure for these larger examples is completely explicit and follows uniformly from the formulas of Section~\ref{sec:rectangular-ears}. In all displayed genuine rear cases, the same pattern persists: degree $2$, a unique triangular closure, and tetrahedral reinforcement of the support edge.

\paragraph{Local stability at larger scale.}
The larger examples reinforce the local picture proved in Section~\ref{sec:rectangular-ears}. The side roots continue to behave as degenerate boundary-transition ears, while every genuinely rear rectangular root in the displayed range carries the same rigid support pattern: degree $2$, a unique triangular closure, and tetrahedral reinforcement of the support edge.

\paragraph{Arithmetic richness versus geometric sparsity.}
The numbers $60$ and $72$ have comparatively rich divisor structure, and correspondingly produce several distinct rectangular roots up to conjugation. Nevertheless, the rectangular family remains sparse even in these cases. Thus an increase in arithmetic richness enlarges the rear marker family, but does not turn it into a dense rear contour at the level of roots alone.

\subsection{Support zones and representative corridors}

The third computational layer shifts from individual ears to interactions between distinct ears.

The natural object here is the support zone
\[
\Sigma_n=G_n\Big[\bigcup_{\rho\in \Rect^\ast(n)}A(\rho)\Big],
\]
together with selected support geodesics between distinct ears.

For the purposes of the present atlas, a representative support-corridor record consists of:
\begin{enumerate}
\item two distinct rectangular ears rooted at $\rho$ and $\sigma$;
\item the support edges $s(\rho)$ and $s(\sigma)$;
\item one chosen support corridor between them;
\item when feasible, a distinction between corridor edges that lie only in triangles and corridor edges that already lie in tetrahedra or higher cliques.
\end{enumerate}

Representative support-corridor data for the first atlas range is recorded in the following table. Here the column ``simplex profile'' gives a coarse qualitative summary of the corridor: ``tetrahedral'' means that every corridor edge already lies in a $4$-clique, ``mixed'' means that different corridor edges lie in cliques of different sizes within the triangular/tetrahedral range, and ``mixed+higher'' indicates that at least one corridor edge lies in a clique of size at least $5$. The remarks column records the corresponding 	extit{edge-clique profile}, namely the sequence of local clique numbers of the successive corridor edges, where each entry is the size of the largest clique containing that edge.

\begin{table}[ht]
\centering
\caption{Representative support-corridor data}
\label{tab:atlas-corridors}
\small
\begin{adjustbox}{max width=\textwidth}
\begin{tabular}{cccccccc}
\toprule
$n$ & $\rho$ & $\sigma$ & $d_{\mathrm{sup}}(\rho,\sigma)$ & chosen endpoints & corridor length & simplex profile & remarks \\
\midrule
8  & $(4,4)$   & $(2^4)$   & 2 & $(4,3,1)\leftrightarrow(3,2,2,1)$ & 2 & mixed & edge profile $4,3$ \\
10 & $(5,5)$   & $(2^5)$   & 4 & $(5,4,1)\leftrightarrow(3,2,2,2,1)$ & 4 & mixed & edge profile $3,4,4,4$ \\
12 & $(6,6)$   & $(4^3)$   & 2 & $(6,5,1)\leftrightarrow(5,4,3)$     & 2 & mixed & edge profile $3,4$ \\
12 & $(6,6)$   & $(3^4)$   & 4 & $(6,5,1)\leftrightarrow(4,3,3,2)$   & 4 & mixed+higher & edge profile $3,4,5,4$ \\
12 & $(6,6)$   & $(2^6)$   & 6 & $(6,5,1)\leftrightarrow(3,2,2,2,2,1)$ & 6 & mixed+higher & edge profile $3,4,4,5,4,4$ \\
12 & $(4^3)$   & $(3^4)$   & 1 & $(4,4,3,1)\leftrightarrow(4,3,3,2)$ & 1 & tetrahedral & edge profile $4$ \\
12 & $(4^3)$   & $(2^6)$   & 4 & $(4,4,3,1)\leftrightarrow(3,2,2,2,2,1)$ & 4 & mixed & edge profile $4,4,4,3$ \\
12 & $(3^4)$   & $(2^6)$   & 2 & $(3,3,3,2,1)\leftrightarrow(3,2,2,2,2,1)$ & 2 & mixed & edge profile $4,3$ \\
\bottomrule
\end{tabular}
\end{adjustbox}
\end{table}

It is also useful to record the basic size and connectivity data of the support zone itself.

\begin{table}[ht]
\centering
\caption{Support-zone size and connectivity}
\label{tab:sigma-connectivity}
\begin{tabular}{ccccc}
\toprule
$n$ & $|\Rect^\ast(n)|$ & $|V(\Sigma_n)|$ & number of components of $\Sigma_n$ & component sizes \\
\midrule
8  & 2 & 4 & 2 & $2,2$ \\
9  & 1 & 2 & 1 & $2$ \\
10 & 2 & 4 & 2 & $2,2$ \\
12 & 4 & 8 & 3 & $4,2,2$ \\
\bottomrule
\end{tabular}
\end{table}

\paragraph{Mixed simplex behavior along corridors.}
The representative support corridors displayed here are not simplex-theoretically uniform. For instance, in $G_{12}$ the corridor between the side ear $(6,6)$ and the genuine rear ear $(3^4)$ exhibits the edge-clique profile
\[
3,4,5,4,
\]
while the corridor from $(6,6)$ to $(2^6)$ exhibits
\[
3,4,4,5,4,4.
\]
Thus the support-corridor geometry is typically mixed, and in the present range it may even enter clique levels higher than the tetrahedral one.

\paragraph{A purely tetrahedral rear corridor.}
At the same time, the atlas also shows that purely tetrahedral rear communication can occur. In $G_{12}$, the support distance between the genuinely rear ears rooted at $(4^3)$ and $(3^4)$ is equal to $1$: their support edges are connected by a single edge, and that edge already lies in a tetrahedral configuration. Thus rear--rear communication need not pass through a lower-dimensional bottleneck.

\paragraph{The role of the support zone.}
The computed examples suggest that the support zone $\Sigma_n$ captures a substantial part of the rear support organization, but not yet a canonical connected rear band. In particular, $\Sigma_8$ and $\Sigma_{10}$ each split into two components, while $\Sigma_{12}$ splits into three. Thus $\Sigma_n$ is already a meaningful intermediate object, but at least in the present range it is not connected in general.

\subsection{Comparison with the established stratifications}

The final computational comparison places the support structures back into the previously established internal morphology of the graph.

The most useful comparisons with the established stratifications are the following:
\begin{enumerate}
\item $\Rect^\ast(n)$ together with the degree layers $D_d(n)$;
\item $\Rect^\ast(n)$ together with the simplex layers $L_r(n)$;
\item support edges and support corridors together with $Ax_n$;
\item support structures together with the thin spine $Sp_n$, when this is visually informative.
\end{enumerate}

These comparisons make the following already established or computationally visible points explicit.

\begin{enumerate}
\item Every rectangular root belongs to
\[
D_2(n)\cap L_2(n).
\]

\item The endpoints of the support edge of every genuine rear ear belong to simplex layers of local simplex dimension at least $3$.

\item The rectangular family intersects the framework only at the even side roots and intersects the self-conjugate axis only at square partitions.

\item The support zone and representative support corridors form a natural interface between the outer marker family and the previously studied inner stratifications of $G_n$.
\end{enumerate}

\paragraph{Outer support geometry versus inner stratifications.}
The overlays with the axis and the previously established inner structures show that the rectangular roots themselves occupy an extreme low-dimensional outer position, while their support structures already enter regions of higher local simplex complexity. In this sense, the support geometry forms an interface between the sparse outer marker family and the more structured internal body of the graph.

This interface is especially relevant for future work on anisotropy, since it begins to reveal how front, side, and rear behavior differ not only visually, but also in the local simplex geometry through which the graph is assembled.

\paragraph{What the atlas does and does not show.}
The computed data suggest that rectangular ears, support edges, and support corridors capture a genuine part of the outer morphology of $G_n$. At the same time, the present computations do not yet identify a canonical rear chain, nor do they provide a complete classification of ear attachments. What they do provide is a stable intermediate language between proved local structure and still-open global rear geometry.

\begin{remark}
The computational atlas is useful for the present paper, but it is not intended to be exhaustive. Its purpose is analytic rather than encyclopedic: each displayed figure and table is included to support a concrete structural point, not merely to display raw data.
\end{remark}

\section{Conclusion and open problems}

We have formalized several canonical pieces of the outer geometry of the partition graph $G_n$. The boundary framework
\[
\mathcal B_n=\mathcal M_n\cup \mathcal L_n\cup \mathcal R_n
\]
provides a front-and-side outer skeleton, while the nontrivial rectangular family
\[
\Rect^\ast(n)=\{(a^b):ab=n,\ a,b\ge 2\}
\]
provides a canonical rear marker family.

The main proved results concern the local structure of rectangular vertices. Every nontrivial rectangular partition $\rho=(a^b)$ has degree $2$, has exactly two explicitly described neighbors, and lies in a unique triangle of $G_n$. This leads to the notions of rectangular ear, attachment pair, and support edge. We also proved that distinct nontrivial rectangular vertices are never adjacent. Thus the weak rectangular contour is not a connected rear contour, but a discrete family of rear markers.

A second proved result concerns the geometry of the support edge. For genuinely rear rectangular ears, namely those with $a,b\ge 3$, the support edge already lies inside a tetrahedral configuration of the clique complex $K_n=\mathrm{Cl}(G_n)$. This yields a first rigorous distinction between side behavior and rear behavior in the outer morphology of $G_n$.

The paper also introduced support zones, support distances, and support corridors as organizing tools for inter-ear communication through the ambient graph. At present, these notions are better understood as a language than as a finished theory. Their global structure remains largely open, and the computational atlas is intended as a first guide to the relevant phenomena rather than as a final classification.

A further feature of the rectangular family is its arithmetic organization. The roots $\rho\in\Rect^\ast(n)$ are naturally indexed by the nontrivial divisors of $n$, while the two antennas correspond to the trivial divisors $1$ and $n$. In the even case, the divisor pair $2,n/2$ yields exactly the two side-ear roots on the boundary framework. Thus the rectangular family is not only a geometric family of rear markers, but also an arithmetic one, linking the outer combinatorial morphology of $G_n$ with the divisor structure of the underlying integer.

Several natural problems remain open.

\begin{enumerate}
\item Is there a canonical rear contour or rear chain that improves on the weak and strong rectangular contours while remaining intrinsic and conjugation-invariant?

\item What is the global structure of the support zone $\Sigma_n$? Is it connected, and if not, how do its components sit relative to the framework and the axis?

\item How large can support distances between distinct ears become, and how do they depend on the factor structure of the corresponding rectangles?

\item What simplex profiles occur along support corridors? Do long rear corridors necessarily pass through lower-dimensional regions, or can purely tetrahedral corridors occur?

\item Can the contour-based attachment language of components of $G_n\setminus C_n$ be reconciled with the local support-edge language in a unified theory of outer appendices and rear attachment loci?

\item More broadly, can one develop an intrinsic theory of front, side, and rear regimes in $G_n$ that interacts naturally with axial distance, simplex stratification, degree concentration, and anisotropy?
\end{enumerate}

In this sense, the present paper should be viewed as a first rigorous step toward a combinatorial theory of the outer boundary of the partition graph. It does not yet provide a complete rear theory, but it isolates canonical outer objects, proves a compact theorem package about them, and supplies a language in which the remaining questions can now be posed precisely.

\section*{Acknowledgements}
The author acknowledges the use of ChatGPT (OpenAI) for discussion, structural planning, and editorial assistance during the preparation of this manuscript. All mathematical statements, proofs, computations, and final wording were checked and approved by the author, who takes full responsibility for the contents of the paper.

\end{document}